\documentclass[12pt]{amsart}

\usepackage{amstext}
\usepackage{amsthm}
\usepackage{amsmath}
\usepackage{latexsym}
\usepackage{amsfonts}
\usepackage{mathtools}
\usepackage[dvipsnames]{xcolor}
\newtheorem*{lemma*}{Lemma}
\newtheorem*{corollary*}{Corollary}
\newtheorem{theorem}{Theorem}
\newtheorem{lemma}[theorem]{Lemma}

\theoremstyle{remark}\newtheorem{remark}[theorem]{Remark}
\theoremstyle{example}

\def\dist{\operatorname{dist}}
\def\capp{\operatorname{cap}}

\newcommand\CC{{\mathbb C}}
\newcommand\RR{{\mathbb R}}
\newcommand\DD{{\mathbb D}}

\newcommand\BB{{\mathbb B_2}}

\begin{document}
	\title[Cyclicity in Dirichlet-type spaces]{Cyclic polynomials in Dirichlet-type spaces in the unit ball of $\mathbb C^2$}
	\author{\L ukasz Kosi\'nski}\email{lukasz.kosinski@uj.edu.pl}
	\author{Dimitrios Vavitsas}\email{dimitris.vavitsas@doctoral.uj.edu.pl}
	\address{Institute of Mathematics, Faculty of Mathematics and Computer Science, Jagiellonian University, \L ojasiewicza 6, 30-348 Krak\'ow, Poland}
	\thanks{Partially supported by NCN grant SONATA BIS no.  2017/26/E/ST1/00723 of the National Science Centre, Poland}
	
	\begin{abstract}
		We characterize polynomials that are cyclic in Dirichlet-type spaces in the unit ball in $\mathbb C^2$.
	\end{abstract}

	\maketitle
	\section{Introduction}
	
	Three classical Hilbert spaces of holomorphic functions in the unit ball of $\mathbb C^n$ are the Hardy, Bergman and Drury-Arveson spaces. All of them are special cases of a space family that depends on a real parameter, called Dirichlet-type spaces. A general introduction to this theory can be found in \cite{Zhu}. 
	
	Our purpose is to characterize the polynomials that are cyclic for the shift operators on these spaces in two variables. An analogous problem for the bidisk was solved in \cite{Benetau} and shortly after extended in \cite{Knese}. A. Sola \cite{Sola} studied this problem in the unit ball. His paper is a main motivation for our research. Sola asked, in particular, for a characterization of cyclic polynomials analogous to that achieved for the bidisc. Note that it looks like a harder problem in the ball because of the absence of determinantal representations. The main aim of the paper is to give an answer to this question.
	
	To attack the problem we shall study the zeros in the sphere of a polynomial non-vanishing in the ball. Using some tools coming from semi-analytic geometry we shall show that this zero set is either finite or contains an analytic curve. The first possibility is in principle simpler to deal with and can be overcome with tools analogous to those used in \cite{Benetau}. If, in turn, the zero set contains an analytic curve, we shall make use of the necessity capacity argument from \cite{Sola} as well as a radial dilation argument. In particular, we shall show a result interesting in its own right: a polynomial $p$ is cyclic if and only if $p/p_r\to 1$, where $p_r(z,w)=p(rz, rw)$ is a radial dilation of $p$. This is connected with the problem of approximating $1/f$ in a space of analytic functions (see \cite{Sol1} for the study of this subject). In the one dimensional case the above-mentioned radial dilation observation was proven in \cite{Knese}. One-variable Dirichlet-type spaces are discussed in the textbook \cite{Blue book}.
	
	\subsection{Dirichlet-type spaces in the unit ball}
	Denote the unit ball by
	$$\BB=\{(z,w)\in \CC^2:|z|^2+|w|^2<1\},$$
	and its boundary, the unit sphere by
	$$\mathbb S_2=\{(\zeta,\eta)\in \CC^2:|\zeta|^2+|\eta|^2=1\}.$$
	Let $f:\BB \rightarrow \mathbb{C}$ be a holomorphic function with power series expansion
	$$f(z,w)=\sum_{k=0}^{\infty}\sum_{l=0}^{\infty}a_{k,l}z^{k}w^{l}.$$
	We say that $f$ belongs to the $\emph{Dirichlet-type}$ $space$ $D_{\alpha}(\BB),$ for a fixed $\alpha\in \mathbb{R},$ if
	\begin{equation}\label{norm with sum}
		|| f||_{\alpha}^{2}=\sum_{k=0}^{\infty}\sum_{l=0}^{\infty}(2+k+l)^{\alpha}\frac{k!l!}{(1+k+l)!}|a_{k,l}|^{2}<\infty.
	\end{equation}
	
	Note that these spaces are Hilbert spaces. The case when $\alpha=0$ corresponds to Hardy spase and $\alpha =-1$ to Bergman. When $\alpha=1$, $D_1(\mathbb B_2)$ coincides with the Drury-Arveson space. The Dirichlet space corresponds to the parameter $\alpha=2.$  A general introduction to function theory in the ball can be found in \cite{Rudin ball} and \cite{Zhu}. Some crucial facts about Dirichlet-type spaces and cyclic vectors in the unit ball can also be found in \cite{Sola}.
	
	The following results from function theory are well known. If $X$ is a normed space and  $C\subset X$ a convex set, then $C$ is closed in norm if and only if it is weakly closed. Moreover, since $D_\alpha(\BB)$ is a reflexible Banach space, given a sequence $\{f_n\}\subset D_\alpha(\BB),$ then $f_n \rightarrow 0$ weakly if and only if $f_n\rightarrow 0$ pointwise and $\sup_n\{||f_n||_\alpha\}<\infty.$
	
	In Dirichlet-type spaces the integral representation of the norm is achieved in a limited range of parameters:
	\begin{lemma}[see\cite{Michalska}]\label{le: equivalent int norm}
		If $\alpha\in (-1,1)$, then $||f||_\alpha$ is equivalent to $$|f|_\alpha:=\int_{\BB} \frac{||\nabla(f)||^2 - |R(f)|^2}{ (1-|z|^2 - |w|^2)^\alpha} dA(z,w).$$
	\end{lemma}
	Above, $\nabla(f)(z,w)=(\partial_zf(z,w),\partial_wf(z,w))$ denotes the \emph{holomorphic gradient} of a holomorphic function $f$ and $$R(f)(z,w)=z\partial_zf(z,w)+w\partial_wf(z,w)$$ is its \emph{radial derivative}. Moreover, $dA(z,w)$ denotes the \emph{normalized area measure}. 
	
	Lemma~\ref{le: equivalent int norm} allows us to deal with Dirichlet norms $D_\alpha$ using analytic methods whenever $\alpha\in (-1,1).$
	
	Polynomials are dense in the spaces $D_{\alpha}(\BB),$ $\alpha \in \RR,$ and $z\cdot f,w\cdot f\in D_\alpha(\BB)$ whenever $f\in D_\alpha(\BB).$ Also, if $\alpha>2$ the spaces $D_\alpha(\BB)$ are algebras, see \cite{Sola}, and, by definition, $D_\alpha(\BB)\subset D_\beta(\BB),$ when $\alpha\geq \beta.$ 
	
	A crucial relation among these spaces is the following:
	\begin{lemma}\label{relation among the spaces}
		Let $f$ be a holomorphic function in $\BB.$ Then
		$$f\in D_\alpha(\BB)\quad  \text{if and only if} \quad  2f+R(f)\in D_{\alpha-2}(\BB).$$
	\end{lemma}
	This elementary observation allows us to use Lemma~\ref{le: equivalent int norm} for a wide range of parameters $\alpha$. 
	
	A \emph{multiplier} of $D_\alpha(\BB)$ is a holomorphic function $\phi:\BB\rightarrow\CC$ that satisfies $\phi\cdot f\in D_\alpha(\BB)$ for all $f\in D_\alpha(\BB).$ The set of all multipliers will be denoted by $M(D_\alpha(\BB)).$ As mentioned above, polynomials, as well as holomorphic functions in a  neighbourhood of the closed unit ball, are multipliers in every space $D_\alpha(\BB)$. 
	
	\subsection{Shift operators and cyclic vectors}
	Consider two bounded linear operators $S_1,S_2:D_{\alpha}(\BB)\rightarrow D_{\alpha}(\BB)$ defined by $S_i:f\mapsto z_if.$ We say that $f\in D_\alpha(\BB)$ is a \emph{cyclic vector} if the closed invariant subspace, i.e.
	$$[f]:=\mathrm{clos}\, \mathrm{span}\{z_1^k z_2^lf:k=0,1,...,l=0,1,...\}$$
	coincides with $D_\alpha(\BB)$ (the closure is taken with respect to the $D_\alpha(\BB)$ norm). In addition, we have the following equivalent definition of cyclicity: $f$ is cyclic if and only if there exist a sequence of polynomials $\{p_n\}$ such that $p_nf\rightarrow 1$ in norm. In other words, $f$ is cyclic if and only if $1\in [f].$
	
	Examples of cyclic functions in various Dirichlet-type spaces  were provided in \cite{Sola}. It is well known (see e.g. \cite{Sola}) that the cyclicity of a function $f\in D_\alpha(\BB)$ is intimately connected with its zero set
	$$\mathcal Z(f) =\{(z,w)\in\CC^2:f(z,w)=0\}.$$
	Since $D_\alpha(\BB)$ enjoys the \emph{bounded point evaluation property} a function that is cyclic cannot vanish inside the unit ball. Any non-zero constant function is cyclic in each space $D_\alpha(\BB).$ Moreover, if $\alpha>2,$ then $f\in D_\alpha(\BB)$ is cyclic precisely when $f$ has no zeros in the closed unit ball. Points lying in the set $\mathcal{Z}(p)\cap \mathbb{S}_2,$ where $p$ is a given polynomial, will be called boundary zeros.
	
	An important result that helps us to restrict the cyclicity problem of polynomials to irreducible ones is the following:
	if $f\in D_\alpha(\BB)$ and $\phi\in M(D_\alpha(\BB)),$ then $\phi f$ is cyclic if and only if both $f$ and $\phi$ are cyclic.
	
	\subsection{Main result}
	Our aim is to characterize the polynomials that are cyclic in Dirichlet-type spaces in the setting of the ball. As we shall see the situation in the ball mirrors the bidisk case, meaning that the cyclicity of a function is inextricably linked to the nature of the boundary zeros. Our goal is to separate the problem into two parts: polynomials with finitely many boundary zeros and polynomials with infinitely many boundary zeros.
	
	The main result is as follows:
	\begin{theorem}\label{main result,theorem}
		Let $p\in \CC[z,w]$ be an irreducible polynomial non-vanishing in the unit ball.
		\begin{enumerate}
			\item If $\alpha\leq 3/2,$ then $p$ is cyclic in $D_\alpha(\BB).$
			\item If $3/2<\alpha\leq 2,$ then $p$ is cyclic in $D_\alpha(\BB)$ if and only if $\mathcal{Z}(p)\cap \mathbb{S}_2$ is empty or finite.
			\item If $\alpha>2,$ then $p$ is cyclic in $D_\alpha(\BB)$ if and only if $\mathcal Z(p)\cap \mathbb{S}_2=\emptyset.$
		\end{enumerate}
	\end{theorem}

	\section{The zero set of a polynomial non-vanishing in the ball}\label{sec:nv}
	
	We begin by studying the boundary zeros of a polynomial non-vanishing in the ball.
	\begin{lemma}\label{zero}
		$\mathcal Z(p)\cap \mathbb{S}_2$ is either a finite set or there is a non-constant analytic curve contained in it.
	\end{lemma}	
	
	It is convenient to look at $\mathcal Z(p)$ $\cap$ $\mathbb{S}_2$ as at a semi-algebraic set.  We are interested in the case where $\mathcal{Z}(p)$ $\cap$ $\mathbb{S}_2$ contains at least one accumulation point.
	
	Recall that a set $A\subset \RR^N$ is said to be \emph{semi-analytic} (resp. \emph{semi-algebraic}), if for any $x\in \RR^N,$ there are a neighbourhood $U=U(x)$ and a finite number of real analytic functions (resp. polynomials) $f_i,$ $g_{ij}$ in $U$ such that
	$$A\cap U=\bigcup\limits_{j=1}^{p}\bigcap\limits_{i=1}^{q}\{x\in U:f_i(x)=0, g_{ij}(x)>0\}.$$
	
	According to this definition, $\mathcal Z(p)\cap \mathbb{S}_2$ is a semi-algebraic (and thus semi-analytic), as it is an intersection of the sphere and two semi-algebraic sets $\{\textrm{Re} (p)=0\}$ and $\{\textrm{Im} (p) =0\}$. Hence Lemma~\ref{zero} is a consequence of the following semi-analytic version of the \emph{Bruhat-Cartan-Wallace Curve Selecting Lemma}: 
	\begin{lemma}[see \cite{Selection lemma}]\label{selecting}
		Let $A$ be a \emph{semi-analytic} set and suppose that $a\in \overline{A\setminus\{a\}}.$ Then there exists an analytic function $\gamma:(0,1)\rightarrow A$ yielding a semi-analytic curve and such that $\lim_{t\rightarrow 0^+}\gamma(t)=a.$
	\end{lemma}
	
	\section{Polynomials with finitely many boundary zeros}\label{sec:f}
	
	The case when a polynomial does not vanish on the closed ball is obvious to deal with. The simplest non-trivial case occurs when $\mathcal Z(p)\cap \mathbb{S}_2$ is finite. This case is relatively simple and can be overcome with tools that were used for Dirichlet-type spaces over the bidisc:
	\begin{theorem}\label{finite}
		Let $p\in \CC[z,w]$ be a polynomial non-vanishing in the unit ball with finitely many zeros in $\mathbb{S}_2.$ Then $p$ is cyclic in $D_\alpha(\BB)$ precisely when $\alpha\leq 2.$
	\end{theorem}
	
	The idea is to compare a polynomial with a product of cyclic polynomials. The function $(z,w)\mapsto |p(z,w)|^2$ is real analytic on $\CC^2,$ and hence, one can apply \emph{\L{}ojasiewicz's inequality} to it on the compact set $\mathbb{S}_2$, see \cite{Lojiasiewicz inequality}. Moreover, there are a constant $C>0$ and a natural number $q$ such that
	$$|p(\zeta,\eta)|\geq C\cdot \dist((\zeta, \eta), \mathcal Z(p)\cap \mathbb{S}_2)^q,$$
	for all $(\zeta,\eta)\in \mathbb{S}_2.$ The distance above is considered with respect to the Euclidean norm.
	
	\begin{proof}[Proof of Theorem~\ref{finite}] Let $p$ be a polynomial, non-vanishing in the ball with finitely many boundary zeros. Let $\mathcal Z(p)\cap \mathbb{S}_2=\{(\zeta_1,\eta_1),...,(\zeta_n,\eta_n)\}.$ Take polynomials $s_1,...,s_n$, which are cyclic in $D_\alpha(\BB)$ precisely when $\alpha\leq 2,$ and such that $\mathcal Z(s_i)\cap \mathbb{S}_2=\{(\zeta_i,\eta_i)\}.$ The polynomials $s_i$ can be trivially constructed as compositions of $\pi(z,w):=1-z$ with unitary matrices $\mathcal{U}_i$ that satisfy $\mathcal{U}_i(\zeta_i,\eta_i)^{T}=(1,0)$, i.e. $s_i = \pi\circ \mathcal U_i$. Since the distance is invariant under unitary transformations, we find that $\dist((\zeta,\eta),(\zeta_i,\eta_i))\geq |\pi\circ\mathcal{U}_i(\zeta,\eta)|=|s_i(\zeta,\eta)|.$ Clearly all $s_j$ are trivially bounded from above on $\mathbb S_2$  by 2, whence $\dist((\zeta,\eta),\mathcal{Z}(p)\cap \mathbb S_2)\geq C_1\prod_{i=1}^{n}|s_i(\zeta,\eta)|,$ for some constant $C_1>0.$
		
		Summing up, by the above-mentioned \L ojasiewicz's inequality there exist a constant $C_2>0$ and $q\in \mathbb N$ such that 
		$$|p(\zeta,\eta)|\geq C_2 \prod_{i=1}^{n}|s_i(\zeta,\eta)|^q,$$
		for all $(\zeta,\eta)\in \mathbb{S}_2.$ Then the rational function $Q$ defined by
		$$Q(z,w)=\frac{\prod_{i=1}^{n}s_i(z,w)^q}{p(z,w)},$$
		is bounded on $\overline{\mathbb{B}}_2.$ Increasing $q\in \mathbb{N}$ we make this function as smooth as we like in  $\mathbb{S}_2.$ In particular, the function $2Q+R(Q)$ lies in the Hardy space for $q$ big enough. Thus, making use of the Lemma~\ref{relation among the spaces} we conclude that $Q$ lives in $D_2(\BB)$. 
		
		The function $pQ$ is cyclic in $D_\alpha(\BB)$, since it is a product of cyclic polynomials, and $pQ\in D_\alpha(\BB),$ for $\alpha\leq 2.$ The assertion follows as $p$ is a multiplier.
	\end{proof}
	
	\section{Cyclicity via radial dilations. Cyclicity for infinitely many boundary zeros.}
	The aim of this section is to prove the following:
	\begin{theorem}\label{main}
		Let $p\in \CC[z,w]$ be a polynomial non-vanishing in the unit ball. Then $p$ is cyclic in $D_\alpha(\BB)$ for any $\alpha\leq 3/2.$
	\end{theorem}
	
	To prove Theorem \ref{main} we shall use \emph{radial dilation} of a function $f:\BB\rightarrow \CC$. It is defined for $r\in(0,1)$ by $f_r(z,w)=f(rz,rw).$  
	To prove Theorem~\ref{main} it is enough to prove the following:
	
	\begin{lemma}\label{missing} If $p\in \CC[z,w]$ does not vanish on $\BB$ and $\alpha\leq 3/2$, then $||p/p_r||_\alpha<\infty$ as $r\to 1^-$.
	\end{lemma}
	Indeed, if Lemma~\ref{missing} holds, then $\phi_r\cdot p\rightarrow 1$ weakly, where $\phi_r:=1/p_r.$ Since $\phi_r$ extends holomorphically past the closed unit ball, $\phi_r$ are multipliers, and hence, $\phi_r\cdot p\in[p].$ Finally, $1$ is weak limit of $\phi_r\cdot p$ and $[p]$ is weakly closed. It is clear that $1\in [p],$ and hence, $p$ is cyclic.
	
	Moreover, it is enough to prove that $||p/p_r||_\alpha<\infty,$ as $r\rightarrow 1^{-},$ for $\alpha=3/2.$ Then the case $\alpha<3/2$ follows since the inclusion $D_{3/2}(\BB)\hookrightarrow D_\alpha(\BB)$ is a compact linear map and weak convergence in $D_{3/2}(\BB)$ gives weak convergence in $D_\alpha(\BB).$
	
	\begin{remark}\label{rem:|fr|-1/2}
		As was mentioned above, it is enough to show that $||p/p_r||_{3/2}< \infty,$ as $r\rightarrow 1^{-}.$ Lemma~\ref{relation among the spaces} shows that this is equivalent to $||2p/p_r+R(p/p_r)||_{-1/2}<\infty,$ as $r\rightarrow 1^{-}.$ The advantage of this approach is that now one can use an equivalent integral norm. In what follows we shall show that $|f_r|_{-1/2}<\infty,$ as $r\rightarrow 1^{-},$ where $f_r$ is one of the following functions: $p/p_r,$ $z\partial_z(p/p_r)$ or $w\partial_w(p/p_r).$ 
	\end{remark}
	
	A standard compact-type argument allows us to estimate the integral norm locally in the sense that it is enough to show that for every point $P\in \mathbb{S}_2$ there exist a neighborhood $U=U(P)$ such that 
	\begin{equation}\label{eq:|f|locally}
		\int_{\BB\cap U} \left(||\nabla(f_r)||^2 - |R(f_r)|^2 \right)\sqrt{1- |z|^2 - |w|^2} dA(z,w)<\infty,
	\end{equation} 
	as $r\to 1^-$, where $f_r\in \{p/p_r,z\partial_z(p/p_r),w\partial_w(p/p_r)\}$.
	
	It is convenient to expand in \eqref{eq:|f|locally} the term $||\nabla(f_r)||^2 - |R(f_r)|^2$ as follows:
	\begin{equation}\label{eq:|f|locally1}||\nabla (f_r)||^2 - |R(f_r)|^2 = ||\nabla(f_r)||^2(1-|z|^2- |w|^2) + |\bar z \partial_wf_r - \bar w \partial_zf_r|^2.
	\end{equation}
	
	Of course, the only problematic points $P$ are those that lie in $\mathcal Z(p)\cap \mathbb{S}_2.$
	
	An idea is to expand $p$ in its Weierstrass form. To do it properly we need to rotate $p$. The following result is clear:
	\begin{remark}\label{rem:U}
		Let $\alpha\in (-1,1).$ Let $\mathcal U$ be a unitary matrix. Then $(f\circ$ $ \mathcal U)_r=f_r\circ$ $\mathcal U$ and $|f\circ$ $\mathcal U|_\alpha = |f|_\alpha$.
	\end{remark}
	
	Let us take a point $P\in \mathcal Z(p)\cap \mathbb{S}_2.$ Composing with a unitary matrix we may assume that $P=(0,1).$ Near $P=(0,1)$ we can expand 
	$$p(z,w)= \alpha(z)(1-h_1(z)w)\cdots (1-h_n(z)w), \quad (z,w)\in \DD(\epsilon)\times \DD(1,\epsilon),$$
	where $\alpha$ is a non-vanishing single-valued holomorphic function in $\DD(\epsilon)$ for some $\epsilon>0$. Here $\DD(\epsilon):=\{z\in \CC:|z|<\epsilon\}$ and $\DD(1,\epsilon):=\{z\in \CC:|z-1|<\epsilon\}$.
	
	The functions $h_j$ are branches of algebraic functions. Sometimes it is convenient to look at them as single-valued holomorphic functions that are well defined on dense simply connected subdomains of $\mathbb D(\epsilon)\setminus\{0\}$, called \emph{slight domains}. Note that this does not affect values of area integrals. See \cite{Knese} for more on the properties of these functions.
	
	In particular, we may define the following function:
	$$H(z,w):=(w-h_1(z))\cdots (w-h_n(z)), \quad (z,w)\in \DD(\epsilon)\times \CC.$$
	The function $H$ is a monic polynomial of degree $n$ with respect to $w,$ whose coefficients  are holomorphic functions in $\DD(\epsilon).$ Recall the following consequence of Puiseux's theorem.
	\begin{corollary*}[see \cite{Puiseux}, Corollary, p.171]
		If $H(z,w)$ is a polynomial with respect to $w\in \CC$ which is monic of degree $m$ and whose coefficients are holomorphic in a neighbourhood of zero in $\CC,$ then there exist an integer exponent $k>0$ and holomorphic functions $\Phi_1,...,\Phi_m$ in the disk $\Delta=\{z\in \CC:|z|<\delta\}$ such that
		$$H(z^k,w)=(w-\Phi_1(z))\cdots (w-\Phi_m(z)) \quad \textrm{in}\quad \Delta\times \CC.$$ 
	\end{corollary*}
	Take $j=1,...,n.$ By Puiseux's theorem there is a non-negative integer $k=k(j)$ and a holomorphic function $\Phi$ such that $\Phi(z^{1/k})=h_j(z)$ for properly chosen branch of $z^{1/k}$, as $z\rightarrow 0.$ The polynomial does not vanish in the ball and therefore $|h_j(z)|^2\leq (1-|z|^2)^{-1}.$ It thus follows from the assumption that 
	\begin{equation}\label{es}|\Phi(z)|^2\leq \frac{1}{1 - |z|^{2k}},
	\end{equation}
	as $z\to 0$. Whence, 
	\begin{equation}\label{problematic h}
		h_j(z)=1+\gamma z^2+o(z^2), \quad z\rightarrow 0
	\end{equation}
	where $|\gamma|\leq 1/2.$
	
	With these tools in hands we are ready to start the proof of the main result of this section. For $r\in (0,1)$ set
	$$q_r(z,w)=\frac{p(z,w)}{p(rz,rw)}=\frac{\alpha(z)}{\alpha (rz)} \prod_{j=1}^{n}\frac{1-h_j(z)w}{1-rh_j(rz)w}.$$
	
	Let us express partial derivatives of $q_r$ in a useful way. Let $$H_j(z,w):=\frac{1-h_j(z)w}{1-rh_j(rz)w}.$$
	\begin{remark}\label{derq}Simple computations give
		\begin{equation*}\label{eq:the term partial z q}
			\partial_zq_r(z,w)=A_1(z,w)+\sum_{j}B_j(z,w)\partial_zH_j(z,w),
		\end{equation*}
		\begin{equation*}\label{eq:the term partial w q}
			\partial_wq_r(z,w)=\sum_{j}B_j(z,w)\partial_wH_j(z,w),
		\end{equation*}
		\begin{equation}\label{eq:the term partial z z q}
			\begin{split}
				\partial_{zz}q_r(z,w)=&A_2(z,w)\\
				&+\sum_{j}\Big(2\Gamma_j(z,w)\partial_zH_j(z,w)+B_j(z,w)\partial_{zz}H_j(z,w)\Big)\\
				&+\sum_{i\neq j}\Delta_{i,j}(z,w)\partial_zH_j(z,w)\partial_zH_i(z,w),
			\end{split}
		\end{equation}
		\begin{equation}\label{eq:the term partial w z q}
			\begin{split}
				\partial_{wz}q_r(z,w)=&\sum_{j}\Big(\Gamma_j(z,w)\partial_wH_j(z,w)+B_j(z,w)\partial_{wz}H_j(z,w)\Big)\\
				&+\sum_{i\neq j}\Delta_{i,j}(z,w)\partial_zH_j(z,w)\partial_wH_i(z,w),
			\end{split}
		\end{equation}
		\begin{equation}\label{eq:the term partial w w q}
			\begin{split}
				\partial_{ww}q_r(z,w)=&\sum_{j}B_j(z,w)\partial_{ww}H_j(z,w)\\
				&+\sum_{i\neq j}\Delta_{i,j}(z,w)\partial_wH_j(z,w)\partial_wH_i(z,w),
			\end{split}
		\end{equation}
		\begin{equation}\label{eq:the term partial z w q}
			\begin{split}
				\partial_{zw}q_r(z,w)=&\sum_{j}\Big(\Gamma_j(z,w)\partial_wH_j(z,w)+B_j(z,w)\partial_{zw}H_j(z,w)\Big)\\
				&+\sum_{i\neq j}\Delta_{i,j}(z,w)\partial_wH_j(z,w)\partial_zH_i(z,w),
			\end{split}
		\end{equation}
		Functions $A_1,A_2,B_j,\Gamma_j,\Delta_{i,j}$ above are products of the terms $H_j(z,w)$ and $\alpha(z)/\alpha(rz),$ as well as derivatives of $\alpha(z)/\alpha(rz).$
	\end{remark}
	
	To simplify the notation all different positive constants that appear in inequalities below and that are uniform with respect to $r\to 1^{-}$ will be denoted by $C>0.$
	
	\begin{remark}\label{re:a(z)/a(rz) and the other are bounded}
		As pointed out before, $\alpha$ does not vanish in $\DD(\epsilon)$, and hence, $A_1$ and $A_2$ are bounded in a neighbourhood of zero, as are the terms $H_j(z,w)$: 
		$$\left|1- H_j(z,w)\right| \leq C \frac{|h_j (z) - rh_j(rz)|}{|1 - rh_j(rz) w|}\leq C \frac{1-r}{|1 - rh_j(rz) w|}\leq C.$$ Consequently, all the functions $A_1,A_2,B_j,\Gamma_j,\Delta_{i,j}$ are bounded.
	\end{remark}
	
	In the sequel the term $\bar z h_j(rz)-r|w|^2h_j'(rz)$ will appear. To estimate it we need the following elementary lemma, due to Z.~B\l ocki:
	\begin{lemma}[\cite{Blocki}, Lemma 3.1]\label{Blocki}
		Let $\Omega$ be a domain in $\RR^n$ and $\psi\in \mathcal C^{1,1}(\Omega)$ be non-negative. Then $\sqrt{\psi}\in \mathcal C^{0,1}(\Omega)$.
	\end{lemma} If $\psi$ is smooth and positive, the above lemma says that $|\partial_{x_j}\psi|^2 \leq C |\psi|$, where $C$ is a locally uniform constant.
	
	We have the following preparatory result:
	\begin{lemma}\label{le:bar z h j(rz) - r|w|^2 h' j(rz)}
		There exists $C>0$ such that
		\begin{equation*}|\bar z h_j(z) - |w|^2 h'_j(z)|\leq C |1 - h_j(z)w|^{1/2},\end{equation*}
		for $(z,w)\in\BB$ and $z$ close to $0$.
		
		\begin{proof}
			Let $\Phi$ holomorphic in a neighborhood of $0$ be such that $h_j(z) = \Phi(z^{1/k})$, for some integer $k\geq 1$, and define $\psi(z)=1-(1-|z|^{2k})|\Phi(z)|^2$. Put 		
			$$\varphi(z)= \frac{\psi(z)}{|z|^{2(k-1)}}.$$ 
			Note that $\varphi\geq 0$ by \eqref{es} and one can check that $\varphi$ is $\mathcal C^{1,1}$ smooth. It follows from Lemma~\ref{Blocki} that $|\partial_{z} \varphi| \leq C \sqrt{\varphi}$. Simple calculations lead to $$|z \partial_z \psi(z) - (k-1) \psi(z)|\leq C |z|^k \sqrt{\psi(z)},$$ 
			and therefore
			\begin{equation}\label{eq:parv}
				|z\partial_z \psi(z)|\leq C(|z|^k \sqrt{\psi(z)} + \psi(z)).
			\end{equation} 
			Since $\psi(z)=O(|z|^{2k})$, one has $\psi(z) \leq C |z|^k \sqrt{\psi(z)}$. Thus, \eqref{eq:parv} implies that
			\begin{equation*} 
				|\partial_z \psi(z)| \leq C |z|^{k-1} \sqrt{\psi(z)}.
			\end{equation*} 
			Computing the derivative of $\psi$ we find that
			$$|k z^{k-1} \bar z^k \Phi(z) - (1-|z|^{2k}) \Phi'(z)|^2 \leq C |z|^{k-1} (1 - (1-|z|^{2k}) |\Phi(z)|^2).$$ From this we immediately get that
			$$|\bar z h_j(z) - (1-|z|^2)h'_j(z)|^2\leq C (1-(1-|z|^2)|h_j(z)|^2),$$ as $z\to 0$
			
			On the other hand
			$|1-h_j(z)w|\geq 1 - |h_j(z)||w|\geq C(1-|h_j(z)|^2 |w|^2)\geq C( 1- (1-|z|^2) |h_j(z)|^2),$ so the assertion follows.
			
		\end{proof}
	\end{lemma} 
	
	\begin{remark}\label{rem}
		Lemma~\ref{le:bar z h j(rz) - r|w|^2 h' j(rz)} applied to points $(rz,rw)$ gives that $$|\bar z h_j(rz) - r|w|^2 h'_j(rz)|\leq C |1 - rh_j(rz)w|^{1/2}$$ for $(z,w)\in \mathbb B_2$, $z$ close to $0$ and $r\to 1^-$. This inequality will play a crucial role in the sequel.
	\end{remark}

	\begin{lemma}\label{lemma: estimate for ||nabla fr||^2}
		The term $||\nabla(q_r)||$ is bounded from above by a constant multiple of the term
		\begin{equation*}
			\sum_{j}\frac{1-r}{|1-rh_j(rz)w|^3}.
		\end{equation*} 
		Moreover, $||\nabla(z \partial_z q_r)||$ and $||\nabla(w\partial_wq_r)||$ are estimated (up to a positive constant) by
		\begin{equation*}
			\sum_{j}\frac{1-r}{|1-rh_j(rz)w|^3}\\
			+\sum_{i\neq j}\frac{(1-r)^2}{|1-rh_j(rz)w|^2|1-rh_i(rz)w|^2}.
		\end{equation*}
		
	\end{lemma}
	\begin{proof}
		Let us start with the first estimate. According to Remark~\ref{derq} it is enough to carry out computation for terms $\partial_z H_j(z,w)$ and $\partial_w H_j(z,w)$. We have
		\begin{equation}\label{eq:H_z}
			\partial_{z}H_j(z,w)=\frac{\sigma(z,w)}{(1-rh_j(rz)w)^2},
		\end{equation}
		where
		\begin{equation}\label{eq:sigma(z,w)}
			\begin{split}
				\sigma(z,w&)=rw(rh_j'(rz)-h_j'(z))
				+wh_j'(z)(r-1)\\
				&+rw^2h_j'(z)(h_j(rz)-h_j(z))+rw^2h_j(z)(h_j'(z)-rh_j'(rz)).
			\end{split}
		\end{equation}
		It is clear from \eqref{problematic h} that \
		\begin{equation*}\label{sigma}|\sigma(z,w)|\leq C(1-r).
		\end{equation*}
		In particular,
		\begin{equation}\label{eq:estimate partial z of 1-hw/1-rhrw}
			|\partial_{z}H_j(z,w)|\leq C(1-r)|1-rh_j(rz)w|^{-3} .
		\end{equation} 
		Similarly,
		\begin{equation*}
			\partial_{w}H_j(z,w)=\frac{rh_j(rz)-h_j(z)}{(1-rh_j(rz)w)^2}.
		\end{equation*} and, by \eqref{problematic h},
		\begin{equation}\label{eq:parw}
			|\partial_wH_j(z,w)|\leq C(1-r)|1-rh_j(rz)w|^{-3} .
		\end{equation} 
		
		To prove the second part of the assertion we need to estimate  $z \partial_{zz} q_r$, $\partial_{wz}q_r,$ $\partial_{zw} q_r$, and $\partial_{ww} q_r$. Let us start with $z\partial_{zz}q_r.$ According to \eqref{eq:the term partial z z q} we have
		\begin{multline*}
			|z\partial_{zz}q_r(z,w)|\leq C\sum_{j}\Big(|\partial_zH_j(z,w)|+|z\partial_{zz}H_j(z,w)|\Big)\\
			+C\sum_{i\neq j}|\partial_{z}H_j(z,w)||\partial_z H_i(z,w)|.
		\end{multline*}
		Thus, what we need to do is to estimate $z\partial_{zz}H_j(z,w).$ From \eqref{eq:H_z} we get
		\begin{equation}\label{eq:zH_zz}
			z\partial_{zz}H_j(z,w)=\frac{z\partial_z\sigma(z,w)}{(1-rh_j(rz)w)^2}+\frac{2r^2wzh_j'(rz)\sigma(z,w)}{(1-rh_j(rz)w)^3},
		\end{equation}
		where
		\begin{equation}\label{eq:z partial z sigma}
			\begin{split}
				&z\partial_z\sigma(z,w)=rw(r^2zh_j''(rz)-zh_j''(z))+zwh_j''(z)(r-1)\\
				&+rzw^2h_j''(z)(h_j(rz)-h_j(z))+rzw^2h_j'(z)(rh_j'(rz)-h_j'(z))\\
				&+rzw^2h_j'(z)(h_j'(z)-rh_j'(rz))+rw^2h_j(z)(zh_j''(z)-r^2zh_j''(rz)).
			\end{split}
		\end{equation}
		It is then clear that 
		\begin{multline*}
			|z\partial_{zz}H_j(z,w)|\leq C\frac{1-r}{|1-rh_j(rz)w|^2}+C\frac{1-r}{|1-rh_j(rz)w|^{3}}\\
			\leq C \frac{1-r}{|1-rh_j(rz)w|^{3}}.
		\end{multline*}
		Thus,
		\begin{align*}
			|z\partial_{zz}q_r(z,w)|\leq& C\sum_{j}\frac{1-r}{|1-rh_j(rz)w|^{3}}\\
			&+C\sum_{i\neq j}\frac{(1-r)^2}{|1-rh_j(rz)w|^2|1-rh_i(rz)w|^{2}}
		\end{align*}
		
		Let us look at $z\partial_{wz}q_r$. According to \eqref{eq:the term partial w z q} we have
		\begin{multline*}
			|z\partial_{wz}q_r(z,w)|\leq C\sum_{j}\Big(|\partial_wH_j(z,w)|+|\partial_{wz}H_j(z,w)|\Big)\\
			+C\sum_{i\neq j}|\partial_zH_j(z,w)||\partial_w H_i(z,w)|.
		\end{multline*}
		The only term that has not been estimated yet is $\partial_{wz}H_j(z,w).$ From \eqref{eq:H_z} we get
		\begin{equation}\label{eq:partial wz H}
			\partial_{wz}H_j(z,w)=\frac{\partial_w\sigma(z,w)}{(1-rh_j(rz)w)^2}+\frac{2rh_j(rz)\sigma(z,w)}{(1-rh_j(rz)w)^3}.
		\end{equation}
		Whence,
		\begin{equation*}
			|\partial_{wz}H_j(z,w)|\leq C(1-r)|1-rh_j(rz)w|^{-3}.
		\end{equation*}
		
		Finally, one can estimates $\partial_{zw}q_r$ and $\partial_{ww}q_r$ in the same way as presented above.
		
	\end{proof}
	
	The next lemma requires more subtle estimations.
	\begin{lemma}\label{lemma: |estimate for bar z fr-bar w fr|^2}
		Let $f_r$ denote one of the functions $q_r,z\partial_zq_r,w\partial_wq_r.$ Then $|\bar z \partial_wf_r - \bar w \partial_zf_r|$ is bounded from above by a constant multiple of
		$$	
		\sum_{j}\frac{1-r}{|1-rh_j(rz)w|^{5/2}} +\sum_{i\neq j}\frac{(1-r)^2}{|1-rh_j(rz)w|^2|1-rh_i(rz)w|^{3/2}}.
		$$
		\begin{proof}
			Let us take $f_r=q_r.$ Since	$|\bar z \partial_wq_r - \bar w \partial_zq_r|$ can be estimated by $||\nabla(q_r)||,$ the assertion for this term follows from Lemma~\ref{lemma: estimate for ||nabla fr||^2}.
			
			Consider the case $f_r=z\partial_zq_r.$ Since
			\begin{equation*}|\bar z \partial_w(z\partial_zq_r) - \bar w \partial_z(z\partial_zq_r)|\leq C ||\nabla(q_r)||+C\left||z|^2\partial_{wz}q_r-\bar wz\partial_{zz}q_r\right|,
			\end{equation*}
			it suffices to estimate $\left||z|^2\partial_{wz}q_r-\bar wz\partial_{zz}q_r\right|.$ According to \eqref{eq:the term partial z z q} and \eqref{eq:the term partial w z q} we get
			\begin{multline*}
				\left||z|^2\partial_{wz}q_r-\bar wz\partial_{zz}q_r\right|\leq C\sum_{j}\Big(|\partial_zH_j(z,w)|+|\partial_wH_j(z,w)|\Big)\\
				+C\sum_{j}\left||z|^2\partial_{wz}H_j(z,w)-\bar w z\partial_{zz}H_j(z,w)\right|\\
				+C\sum_{i\neq j}|\partial_zH_j(z,w)|\left||z|^2\partial_{w}H_i(z,w)-\bar w z\partial_{z}H_i(z,w)\right|.
			\end{multline*}
			Note that the terms $\partial_zH_j(z,w),$ $\partial_wH_j(z,w)$ have already been estimated in \eqref{eq:estimate partial z of 1-hw/1-rhrw}, \eqref{eq:parw}. 
			
			By \eqref{eq:zH_zz} and \eqref{eq:partial wz H}, to deal with $|z|^2\partial_{wz}H_j(z,w)-\bar w z\partial_{zz}H_j(z,w)$ it suffices to estimate
			$$\frac{|z|^2\partial_w\sigma(z,w)-\bar w z\partial_z\sigma(z,w)}{(1-rh_j(rz)w)^2} \text{ and } \frac{2r\sigma(z,w) z(\bar z h_j(rz)-r|w|^2h_j'(rz))}{(1-rh_j(rz)w)^3}.$$
			By \eqref{eq:sigma(z,w)}, \eqref{eq:z partial z sigma} and \eqref{problematic h} one has $|\sigma(z,w)|,$ $|\partial_z\sigma(z,w)|,$ $|\partial_w\sigma(z,w)|\leq C(1-r).$ Thus, $\left||z|^2\partial_{wz}H_j(z,w)-\bar w z\partial_{zz}H_j(z,w)\right|$ is bounded by a constant times
			$$\frac{1-r}{|1-rh_j(rz)w|^2}+\frac{(1-r)|\bar z h_j(rz)-r|w|^2h_j'(rz)|}{|1-rh_j(rz)w|^3}.$$
			Applying Remark~\ref{rem} to the second term above we get
			\begin{multline*}
				\left||z|^2\partial_{wz}H_j(z,w)-\bar w z\partial_{zz}H_j(z,w)\right|\leq \frac{C(1-r)}{|1-rh_j(rz)w|^2}\\
				+\frac{C(1-r)}{|1-rh_j(rz)w|^{5/2}}.
			\end{multline*}
			
			Let us focus on $|z|^2\partial_{w}H_i(z,w)-\bar w z\partial_{z}H_i(z,w).$ Carrying out some computations we get that 
			\begin{multline*}
				|z|^2\partial_{w}H_i(z,w)-\bar w z\partial_{z}H_i(z,w)=\frac{rz|w|^2(h_i'(z)-rh_i'(rz))(1-wh_i(z))}{(1-rh_i(rz)w)^2}\\
				+\frac{(1-r)z(|w|^2h_i'(z)-\bar zh_i(z))}{(1-rh_i(rz)w)^2}
				+\frac{r|z|^2(h_i(rz)-h_i(z))(1-wh_i(z))}{(1-rh_i(rz)w)^2}\\
				+\frac{rzw(h_i(rz)-h_i(z))(\bar zh_i(z)-|w|^2h_i'(z))}{(1-rh_i(rz)w)^2}.
			\end{multline*}
			Note that $1-wh_i(z)=1-rh_i(rz)w+w(rh_i(rz)-h_i(z)).$ It follows from Lemma~\ref{le:bar z h j(rz) - r|w|^2 h' j(rz)} and \eqref{problematic h} that
			\begin{equation*}\label{eq:|z|^2 partial w H i(z,w)-bar w z partial z H i(z,w)}
				\left||z|^2\partial_{w}H_i(z,w)-\bar w z\partial_{z}H_i(z,w)\right|\leq C\frac{1-r}{|1-rh_i(rz)w|^{3/2}}.	
			\end{equation*}

			Finally, take $f_r=w\partial_wq_r.$ Then $|\bar z \partial_w(w\partial_wq_r) - \bar w \partial_z(w\partial_wq_r)|$ is bounded by a constant times 
			$||\nabla(q_r)||+\left||w|^2\partial_{zw}q_r-\bar zw\partial_{ww}q_r\right|.$ According to \eqref{eq:the term partial w w q}  and \eqref{eq:the term partial z w q} we get
			\begin{multline*}
				\left||w|^2\partial_{zw}q_r-\bar zw\partial_{ww}q_r \right|\leq C\sum_{j}|\partial_wH_j(z,w)|\\
				+C\sum_{j}\left||w|^2\partial_{zw}H_j(z,w)-\bar z w\partial_{ww}H_j(z,w)\right|\\
				+C\sum_{i\neq j}|\partial_wH_j(z,w)|\left||w|^2\partial_{z}H_i(z,w)-\bar z w\partial_{w}H_i(z,w)\right|.
			\end{multline*}
			Let us expand
			\begin{multline*}
				|w|^2\partial_{z}H_i(z,w)-\bar z w\partial_{w}H_i(z,w)=\\ \frac{rw|w|^2(rh_i'(rz)-h_i'(z))(1-wh_i(z))}{(1-rh_i(rz)w)^2}
				+\frac{(1-r)w(\bar zh_i(rz)-|w|^2h_i'(z))}{(1-rh_i(rz)w)^2}+\\
				\frac{w^2(h_i(rz)-h_i(z))(r|w|^2h_i'(z)-\bar zh_i(z))}{(1-rh_i(rz)w)^2}+\\
				\frac{\bar z w(h_i(rz)-h_i(z))(wh_i(z)-1)}{(1-rh_i(rz)w)^2}.
			\end{multline*}
			
			In particular, the following estimation holds:
			\begin{equation*}
				\left||w|^2\partial_{z}H_i(z,w)-\bar z w\partial_{w}H_i(z,w)\right|\leq C\frac{1-r}{|1-rh_i(rz)w|^{3/2}}. 
			\end{equation*}
			Similarly, one can show that
			\begin{equation*}
				\Big||w|^2\partial_{zw}H_j(z,w)-\bar z w\partial_{ww}H_j(z,w)\Big|\leq C\frac{1-r}{|1-rh_j(rz)w|^{5/2}}.
			\end{equation*}
		\end{proof} 
	\end{lemma} 
	
	Our goal is to prove \eqref{eq:|f|locally}. Thanks to Lemma~\ref{lemma: estimate for ||nabla fr||^2}, Lemma~\ref{lemma: |estimate for bar z fr-bar w fr|^2} and \eqref{eq:|f|locally1} it is suffices to show that 
	\begin{equation}\label{int1}\int_{\BB\cap U_j}\frac{(1-r)^2 (1- |z|^2 - |w|^2)^{\alpha}}{|1 - rh_j(rz)w |^{\beta}} dA(z,w)<\infty, \text{ as }r\to 1^-,\end{equation}where $(\alpha,\beta)=(3/2,6),$ or $(\alpha, \beta)=(1/2,5)$, and
	\begin{equation*}\int_{\BB\cap U_{i,j}}\frac{(1-r)^4 (1- |z|^2 - |w|^2)^{\alpha'}}{|1 - rh_j(rz)w |^{\beta'}|1 - rh_i(rz)w|^{\gamma'}} dA(z,w)<\infty, \text{ as } r\rightarrow 1^{-},
	\end{equation*}
	where $(\alpha',\beta',\gamma')=(3/2,4,4),$ or $(\alpha', \beta', \gamma')=(1/2,4,3).$ Applying H\"{o}lder's inequality to the second integral above we see that it is enough to prove \eqref{int1}.
	
	The following integral estimate of Forelli and Rudin is crucial for our computations:
	\begin{remark}\label{Forelli rudin}[see \cite{Forelli Rudin}, Theorem 1.7]
		Let $a\in (-1,\infty)$ and $b \in (0,\infty).$ Then
		\begin{equation}
			\int_{\DD}\frac{(1-|w|)^{a}}{|1-\overline{z} w|^{2+a+b}}dA(w)\asymp (1-|z|^2)^{-b},
		\end{equation}
		as $|z|\rightarrow 1^{-}.$
	\end{remark}
	
	\begin{lemma}\label{lem:ab} Let $\alpha\in (-1,\infty)$ and $\beta-2-\alpha \in (0,\infty).$ Set $U_j=\Omega_j\times \DD(1,\epsilon),$ where $\Omega_j$ is the slight domain on which $h_j$ is defined. Then
		\begin{align*}
			\int_{\BB\cap U_j}\frac{(1- |z|^2 - |w|^2)^\alpha}{|1 - rh_j(rz)w |^\beta}& dA(z,w)\\
			&\leq C\int_{\Omega_j} \frac{1}{ (1- r^2 |h_j(rz)|^2 (1-|z|^2))^b} dA(z),
		\end{align*}
		where $b=\beta-\alpha-2.$ 
		\begin{proof}
			Let us estimate
			\begin{align*}
				&\int_{\BB\cap U_j}\frac{(1- |z|^2 - |w|^2)^\alpha}{|1 - rh_j(rz)w |^\beta} dA(z,w)\\
				&\leq  \int_{\Omega_j}(1-|z|^2)^{\alpha} \int_{\{|w|<\sqrt{1-|z|^2}\}} \frac{\left(1 - |(1-|z|^2)^{-1/2}w|^2\right)^\alpha}{ |1 - rh_j(rz) w|^\beta} dA(w)dA(z)\\
				&\leq \int_{\Omega_j}(1-|z|^2)^{\alpha+2} \int_{\DD} \frac{(1 - |w|^2)^\alpha}{ |1 - rh_j(rz) \sqrt{1-|z|^2}w|^\beta} dA(w)dA(z)\\
				&\leq
				C \int_{\Omega_j} \frac{1}{ (1- r^2 |h_j(rz)|^2 (1-|z|^2))^b} dA(z).
			\end{align*}
			In the computations carried out above the change-of-variables $w\mapsto w(1-|z|^2)^{-\frac{1}{2}}$ and Remark~\ref{Forelli rudin} have been used.
		\end{proof}
	\end{lemma} 	
	
	Set $h_j=h$ and $\Omega_j=\Omega.$ According to Lemma~\ref{lem:ab}, to get \eqref{int1}, it suffices to show that the integral
	$$\int_{\Omega}\frac{(1-r)^2}{(1-r^2(1-|z|^2)|h(rz)|^2)^{5/2}}dA(z)$$
	is bounded as $r\rightarrow 1^{-}.$ 	
	
	We showed in \eqref{problematic h} that $h(z)=1+\gamma z^2+o(z^2),$ where $|\gamma|\leq 1/2.$ If $|\gamma|<1/2$, then $|h(z)|^2\leq 1+|z|^2,$ and we end up with the integral
	$$\int_{0}^{\epsilon}\frac{(1-r)^2x}{(1-r^2+r^2x^4)^{5/2}}dx,$$
	which is bounded as $r\rightarrow 1^{-}.$
	
	The case $|\gamma|=1/2$ needs some more attention. Without loss of generality we may then assume that $\gamma=1/2,$ and hence it suffices to consider
	$$ \int_{\Omega}\frac{(1-r)^{2}}{(1-r^2+r^2(|z|^2-r^2Re(z^2)+o(z^2)))^{5/2}}dA(z),$$
	as $r\rightarrow 1^-.$ Note that $z\mapsto o(z^2)$ is a real-valued function, smooth on $\Omega$. Setting $z=x+iy,$ it is enough to estimate the integral
	\begin{equation*}
		\iint_{\Omega} \frac{(1-r)^{2}}{(1-r^2+r^2((1-r^2)x^2+(1+r^2)y^2+o(z^2)))^{5/2}}dxdy.
	\end{equation*}
	Fix $r\in (0,1)$ and define $u$ in $\DD$ by
	$$u(x,y)=\frac{(1-r)^2}{(1-r^2+r^2((1-r^2)x^2+(1+r^2)y^2))^{5/2}}\geq 0.$$
	Next, define $G\in C^\infty(\Omega)$ by
	$$G(x,y)=\Big(x, y\sqrt{1+\frac{o(z^2)}{1+r^2}}\Big).$$
	Note that the Jacobian of $G$ at $(x,y)$ is close to $1$ as $(x,y)$ is close to zero. This allows us to make a proper change of variables. Therefore it is enough to estimate $\iint_{\DD}u(x,y)dxdy$ which, in turn, boils down to
	$$\int_{0}^{1}\int_{0}^{2\pi}\frac{(1-r)^{2}t}{\Big(1-r^2+r^2t^2(1-r^2\cos(\theta))\Big)^{5/2}}dtd\theta.$$
	It is elementary to see that the last integral is bounded as $r\rightarrow 1^-$, so the desired result follows.
	
	\section{Non-cyclicity for infinitely many boundary zeros}
	All that is left to prove that any polynomial non-vanishing in the ball that vanishes along a curve in the sphere fails to be cyclic in certain $D_{\alpha}(\BB).$ We have the following:
	
	\begin{theorem}\label{non-cyclic}
		Let $p\in \CC[z,w]$ be a polynomial non-vanishing in the unit ball with infinitely many zeros in $\mathbb{S}_2.$ Then $p$ is not cyclic in $D_\alpha(\BB)$ whenever $\alpha>3/2.$
	\end{theorem}
	
	This result is a consequence of capacitary conditions and the nature of the boundary zeros of a polynomial non-vanishing in the ball. We consider Riesz $\alpha$-capacity for a fixed $\alpha\in(0,2]$ with respect to the \emph{anisotropic distance} in $\mathbb S_2$ given by 
	$$d(\zeta,\eta)=|1-\langle \zeta,\eta \rangle|^{1/2}$$
	and the positive kernel $K_\alpha:(0,\infty)\rightarrow [0,\infty)$ given by
	
	$$K_\alpha(t)=
	\begin{dcases}
		t^{\alpha-2}, & \alpha\in(0,2) \\
		\log(e/t), & \alpha=2
	\end{dcases}
	.$$
	
	If $\mu$ is a Borel probability measure supported on some Borel set $E\subset \mathbb S_2,$ then the \emph{Riesz} $\alpha$-\emph{energy} of $\mu$ is given by
	$$I_\alpha[\mu]=\iint_{\mathbb S_2}K_\alpha(|1-\langle \zeta,\eta \rangle|)d\mu(\zeta)d\mu(\eta)$$
	and the \emph{Riesz} $\alpha$-\emph{capacity} of $E$ by
	$$\text{cap}_\alpha(E)=1/\inf\{I_\alpha[\mu]:\mu\in \mathcal{P}(E)\},$$
	where $\mathcal{P}(E)$ is the set of Borel probability measures supported on $E.$
	
	The following theorem is crucial to identify non-cyclicity in Dirichlet-type spaces:
	
	\begin{theorem}[\cite{Sola}, Remark 4.5.]\label{sola}
		Let $\alpha\in \RR$ be fixed. Suppose $f\in D_\alpha(\BB)$ has $\capp_\alpha(\mathcal{Z}(f^*))>0.$ Then $f$ is not cyclic in $D_\alpha(\BB).$ 
	\end{theorem}
	
	Note that $\mathcal{Z}(f^*)$ is the zero set in the sphere of the radial limits of $f.$ In particular, $\mathcal Z(p)\cap \mathbb{S}_2$ coincides with $\mathcal{Z}(p^*)$ when $p$ is a polynomial.
	
	\begin{proof}[Proof of Theorem~\ref{non-cyclic}]
		Let $p$ be a polynomial non-vanishing in the ball that has infinitely many boundary zeros. According to Lemma~\ref{zero}, the set of the boundary zeros contains at least one non-constant analytic curve. The model polynomial $1-2zw$ vanishes on an analytic curve $\mathcal Z(1-2zw) \cap \mathbb S_2 = \{(e^{i\theta}/\sqrt{2}, e^{-i\theta}/\sqrt{2}): \theta \in \mathbb R\}$ and satisfies $\text{cap}_\alpha(\mathcal Z(1-2zw)\cap \mathbb S_2)>0,$ when $\alpha>3/2,$ see \cite{Sola}. 
		
		Let $\gamma:(-\epsilon, \epsilon)\rightarrow \mathcal{Z}(p)\cap \mathbb S_2$ and $\omega:(-\epsilon, \epsilon)\rightarrow \mathcal{Z}(1-2zw)\cap \mathbb S_2$ be analytic functions yielding the above-mentioned curves. Of course, $\omega$ can be written with an explicit formula $\omega(\theta) = (e^{i\theta} /\sqrt2, e^{-i\theta}/\sqrt2)$. Note that $d(\omega(\theta), \omega(\theta'))\asymp |\theta-\theta'|$, where $d$ denotes the anisotropic distance in $\mathbb S_2.$ The same distance estimate is true for $\gamma$, namely $d(\gamma(\theta), \gamma(\theta'))\asymp |\theta-\theta'|$. To see it write $\gamma=(\gamma_1,\gamma_2)$. We can make, for the sake of simplicity, some additional assumptions: $\theta'=0$, $\gamma(0)= (0,1)$ (compose with a unitary matrix) and $\gamma_1'(0)\neq 1$. It follows from \eqref{problematic h} that $1 - \gamma_2(t) = O(\gamma_1^2(t)).$ Let us write $\gamma_1(t) = a_1 t + O(t^2),$ $\gamma_2(t) = 1+b_2 t^2 + O(t^3)$. Since $1 = |\gamma_1(t)^2| + |\gamma_2(t)|^2$, $t\in (-\epsilon, \epsilon)$, we get that $b_2\neq 0$. Then $d(\gamma(0), \gamma(\theta)) = |1-\gamma_2(\theta)|^{1/2} \asymp |\theta|$, as claimed.
		
		In particular, $T:=\omega\circ\gamma^{-1}$ transforms the analytic curve contained in $\mathcal{Z}(p)\cap \mathbb S_2$ into the one contained in $\mathcal{Z}(1-2zw)\cap \mathbb S_2$ so that $K_\alpha(|1- \langle T(\zeta), T(\eta)\rangle |) \asymp K_\alpha (1 - \langle \zeta, \eta \rangle )$. Now we can proceed as in \cite[Theorem 5.3.1]{Ransford}. Precisely, it follows from \cite[Theorem A.4.4]{Ransford} that for a probability measure supported in $\omega((-\epsilon,\epsilon))$ the measure $\nu=\mu T^{-1}$ is supported in $\gamma((-\epsilon, \epsilon))$. In particular, $I_\alpha[\mu]\asymp I_\alpha[\nu]$ for $\alpha>3/2$. From this we deduce that $\capp_\alpha(\gamma([0,1]))>0$ whenever $\capp_\alpha(\omega([0,1]))>0.$ Consequently, $\capp_\alpha(\mathcal{Z}(p)\cap \mathbb S_2)>0$ for $\alpha>3/2.$ 
		
		The assertion thus follows from Theorem~\ref{sola}.
	\end{proof}

	\begin{remark}\label{rem:final}
		Many of ideas presented above may be extended to Dirichlet-type spaces of $\mathbb B_n$ with $n\geq 3$. One can expect that the solution of this general problem depends on $\dim_{\mathbb R}(\mathcal Z(p)\cap\mathbb S_n).$ 
		
		Let us mention here that some partial results for the polydisc $\mathbb D^n$ with $n\geq 3$ were obtained in \cite{Ber}. A crucial difference between the polydisc and the ball lies in their geometry. They are not biholomorphic to each other, even their group of automorphisms (both transitive) are not isomorphic. What played a key role in studying the cyclicity of a polynomial in Dirichlet-type spaces over the bidisc (see \cite{Benetau, Knese}) was the size of its zero set on the Shilov (distinguished) boundary. Recall that the Shilov boundary of a given domain is the smallest subset of its topological boundary where an analog of the maximum modulus principle holds. Let us mention that it shows yet another important difference between these domains: the distinguished boundary of the ball coincides with its topological boundary, which is not the case for the polydisc -- its distinguished boundary is the torus.
		
		It is worth mentioning that Theorem~\ref{main result,theorem} has a simpler formulation on $\mathbb B_2$ comparing to its counterpart on $\mathbb D^2$. This is because in the bidisc one needs to treat $(1-z)$ and $(1-w)$ separately: they are cyclic even thought their zero sets $\mathcal Z(p)\cap \partial\mathbb D^2$ are quite large. 
	\end{remark}
	
	{\bf Acknowledgements.} We would like to thank the anonymous referees for numerous remarks that substantially improved the shape of the paper.


\begin{thebibliography}{9}
		\bibitem{Benetau} C. B\'{e}n\'{e}teau, G. Knese, \L{}. Kosi\'{n}ski, C. Liaw, D. Seco, and A. Sola, Cyclic polynomials in two variables, Trans. Amer. Math. Soc. \textbf{368} (2016), 8737--8754.
		\bibitem {Ber} Bergqvist, L. A note on cyclic polynomials in polydiscs, Anal. Math. Phys. \textbf{8}, 197–-211 (2018).
		\bibitem{Blocki} Z. B\l ocki, Regularity of the degenerate Monge-Amp\`ere equation on compact K\"ahler manifolds, Mathematische Zeitschrift \textbf{244} (2003), 153--161.
		\bibitem{Selection lemma} Z. Denkowska, M. Denkowski, Along and winding road to definable sets, Journal of Singularities, \textbf{13} (2015), 57--86.
		\bibitem{Blue book} O. El-Fallah, K. Kellay, J. Mashreghi, and T. Ransford, A primer on the Dirichlet space, Cambridge Tracts in Mathematics \textbf{203}, Cambridge University Press, 2014.
		\bibitem{Forelli Rudin} H. Hedenmalm, B. Korenblum, K. Zhu, Theory of Bergman spaces, Springer, New York, 2000.
		\bibitem{Knese} G. Knese, \L{}. Kosi\'{n}ski, T.J. Ransford, and A.A. Sola, Cyclic polynomials in anisotropic Dirichlet spaces, J. Anal. Math. \textbf{138} (2019), 23--47.
		\bibitem{Lojiasiewicz inequality} S.G. Krantz and H.R. Parks, A primer of real analytic functions, 2nd edition, Birkh\"{a}user Advanced Texts: Basler Lehrb\"{u}cher. Birkh\"{a}user Boston, Inc., Boston, MA, 2002.
		\bibitem{Puiseux} S. \L{}ojasiewicz, Introduction to complex analytic geometry, Translated from Polish by Maciej Klimek, Birkh\"{a}user Verlag, Basel, 1991.
		\bibitem{Michalska} M. Michalska, On Dirichlet type spaces in the unit ball of $\CC^n,$ Ann. Univ. Mariae Curie-Sk\l{}odowska Sect. A \textbf{65} (2011), 78--86.
		\bibitem{Ransford} T. Ransford, Potential theory in the complex plane, London Mathematical Society Student Texts \textbf{28}, Cambridge University Press, Cambridge, 1995.
		\bibitem{Rudin ball} W. Rudin, Function theory in the unit ball of $\CC^n,$ Springer, Grundlehren der mathematischen Wissenschaften \textbf{241}, Springer, New York, 1980. 
		\bibitem{Sol1} M. Sargent, A. Sola, Optimal approximants and orthogonal polynomials in several variables, Canad. J. Math. (2020), 1--29.	
		\bibitem{Sola} A. Sola, A note on Dirichlet-type spaces and cyclic vectors in the unit ball of $\CC^2$, Arch. Math. \textbf{104} (2015), 247--257.
		\bibitem{Zhu} K. Zhu, Spaces of holomorphic functions in the unit ball, Graduate Texts in Mathematics \textbf{226}, Springer, New York, 2005.
	\end{thebibliography}
\end{document}